\documentclass[12pt,leqno]{article}

\usepackage{amssymb, amsmath, amsthm}
\usepackage{cases}
\usepackage{braket}

\newtheorem{theorem}{Theorem}[section]
\newtheorem{proposition}[theorem]{Proposition}
\newtheorem{lemma}[theorem]{Lemma}

\theoremstyle{definition}

\newtheorem{example}[theorem]{Example}

\theoremstyle{remark}
\newtheorem{remark}[theorem]{Remark}

\DeclareMathOperator{\sign}{sign}

\numberwithin{equation}{section}

\title{On cell problems for Hamilton-Jacobi equations with non-coercive Hamiltonians and its application to homogenization problems}
\date{}
\author{
Nao Hamamuki\thanks{Department of Mathematics, Hokkaido University,
Kita 10, Nishi 8, Kita-Ku, Sapporo, Hokkaido, 060-0810, Japan, e-mail:hnao@math.sci.hokudai.ac.jp},
\and 
Atsushi Nakayasu$^{\dagger}$,
\and 
Tokinaga Namba\thanks{Graduate School of Mathematical Sciences, University of Tokyo, 3-8-1 Komaba, Meguro-ku, Tokyo, 153-8914, Japan, e-mail:ankys@ms.u-tokyo.ac.jp (Atsushi Nakayasu), namba@ms.u-tokyo.ac.jp (Tokinaga Namba)}}

\begin{document}

\maketitle



%

\begin{abstract}
We study a cell problem arising in homogenization for a Hamilton-Jacobi equation whose Hamiltonian is not coercive.
We introduce a generalized notion of effective Hamiltonians by approximating the equation and characterize the solvability of the cell problem in terms of the generalized effective Hamiltonian. 
Under some sufficient conditions, the result is applied to the associated homogenization problem. 
We also show that homogenization for non-coercive equations fails in general.
\end{abstract}




%





\section{Introduction}

We consider a Hamilton-Jacobi equation of the form
\begin{equation*}
\tag{CP}
\label{cp}
H(x,Du(x)+P)=a \quad \text{in $\mathbf{T}^{N}$}
\end{equation*}
and study a problem to find, for a given $P\in\mathbf{R}^{N}$,
a pair of a function $u \colon \mathbf{T}^{N} \to \mathbf{R}$ and a constant $a \in \mathbf{R}$ 
such that $u$ is a Lipschitz continuous viscosity solution of \eqref{cp}.
Here, $\mathbf{T}^{N}:=\mathbf{R}^{N}/\mathbf{Z}^{N}$ and a function $u$ on $\mathbf{T}^{N}$ is regarded as a function defined on $\mathbf{R}^{N}$ with $\mathbf{Z}^{N}$-periodicity,
i.e., $u(x+z)=u(x)$ for all $x\in\mathbf{R}^{N}$ and $z\in\mathbf{Z}^{N}$.
Moreover, $Du$ denotes the gradient, i.e., $Du=(\partial u/\partial x_{1},\cdots,\partial u/\partial x_{N})$.
This kind of problem is called a \emph{cell problem} in the theory of homogenization.
The constant $a$ satisfying \eqref{cp} is called a \emph{critical value}
if it is uniquely determined.

As a typical example in this paper, we consider the Hamiltonian $H:\mathbf{T}^{N}\times\mathbf{R}^{N}\to\mathbf{R}$ in \eqref{cp} given by 
\begin{equation}
\label{eq:ourH}
H(x,p)=\sigma(x)m(|p|),
\end{equation}
where $\sigma$ and $m$ satisfy
\begin{itemize}
\item[(H1)]
$\sigma \colon \mathbf{T}^{N}\to(0,\infty)$ is a continuous function,
\item[(H2)]
$m \colon [0,\infty)\to(0,1)$ is a Lipschitz continuous function,
\item[(H3)]
$m$ is strictly increasing and $m(r)\to1$ as $r\to\infty$.
\end{itemize}
Due to the boundedness of $m$,
our cell problem does not necessarily admit a solution $(u,a)$,
and the solvability depends on $P\in\mathbf{R}^{N}$.
One of goals in this paper is to characterize the set of $P\in\mathbf{R}^{N}$ such that the cell problem admits a solution.
The other goal is to apply the result to the associated homogenization problem.

A result for existence of a solution of cell problems for Hamilton-Jacobi equations
was first established by Lions, Papanicolaou and Varadhan \cite{LPV} under the assumption that the Hamiltonian is \emph{coercive}, i.e.,
\begin{equation}
\label{eq:coercive}
\lim_{r\to\infty}\inf\{H(x,p) \mid x\in\mathbf{T}^{N},p\in\mathbf{R}^{N},\ |p|\ge r\}=+\infty.
\end{equation}
Their method begins with considering the following approximate equation with a parameter $\delta>0$:
\begin{equation}
\label{eq:dcp}
\delta u_{\delta}(x)+H(x,Du_{\delta}(x)+P)=0 \quad \text{in $\mathbf{T}^{N}$}.
\end{equation} 
By a standard argument of viscosity solutions, it turns out that there exists a unique solution $u_{\delta}$ and that a family of functions $\{\delta u_{\delta}\}_{\delta>0}$ is uniformly bounded. Thus, (formally) $\{Du_{\delta}\}_{\delta>0}$ is uniformly bounded thanks to the coercivity. Therefore, by taking a subsequence if necessary, $\delta u_{\delta}$ and $u_{\delta}-\min u_{\delta}$ uniformly converge to a constant $-a$ and a function $u$ as $\delta\to0$, respectively.
A stability argument of viscosity solutions shows that $u$ and $a$ solve \eqref{cp}. For more details, see \cite{LPV} and \cite{E}.
We point out that the paper \cite{E} also studies second order uniformly elliptic equations by using a similar argument.

Unfortunately, our Hamiltonian \eqref{eq:ourH} is not coercive because of the boundedness of the function $m$.
When a Hamiltonian is not coercive, the method of \cite{LPV} becomes very delicate.
Cardaliaguet \cite{C} shows, in fact, that $\delta u_{\delta}$ may not converge to a constant; this result does not cover our setting.
We also refer the reader to \cite{AL} as a related work to \cite{C}.
Homogenization results with non-coercive Hamiltonians can be seen in \cite{BT, B, BW, CLS, CNS, Go, S}.
Hamiltonians with some partial coercivity is studied in \cite{BT},
and \cite{B} treats equations with $u/\varepsilon$-term.
The papers \cite{BW, S, Go} are concerned with homogenization on spaces with a (sub-Riemannian) geometrical condition.
The authors of \cite{CLS} study moving interfaces with a sign changing driving force term
while \cite{CNS} considers $G$-equations being possibly non-coercive.
Homogenization for degenerate second order equations has been developed by \cite{AB, CM}.
Our Hamiltonian \eqref{eq:ourH} has not been treated yet in the context of homogenization.

We now present our main results and briefly explain our approach for the non-coercive Hamilton-Jacobi equation \eqref{cp}.
Let us consider an approximate equation of the form
\begin{equation*}
\tag{CP$_{n}$}
\label{CPn}
H_{n}(x,Du_{n}(x)+P)=\bar{H}_{n}(P) \quad \text{in $\mathbf{T}^{N}$}
\end{equation*}
for each $n\in\mathbf{N}$.
Here $\{ H_{n} \} _{n \in \mathbf{N}}$ is a family of coercive Hamiltonians which approximate $H$.
For the detailed assumptions, see (A1)--(A4) in Section 3.
By the coercivity of $H_{n}$, the result of \cite{LPV} ensures that, for each $n\in\mathbf{N}$, the approximate equation has a solution $(u_n, \bar{H}_{n}(P))$ for every $P\in\mathbf{R}^{N}$.
The function $\bar{H}_{n}(\cdot)$ is called an \emph{effective Hamiltonian}, which appears in a limit equation in homogenization problems (see \cite{LPV}).
Our first main result is that, for each $P\in\mathbf{R}^{N}$, there exists a limit $\bar{H}_{\infty}(P)$ of $\bar{H}_{n}(P)$ as $n\to\infty$ and its value is independent of approximation (Theorem \ref{thm:1}).
In this paper we call $\bar{H}_{\infty}(\cdot)$ a \emph{generalized effective Hamiltonian},
which is defined on the whole of $\mathbf{R}^{N}$ even if \eqref{cp} is not solvable for some $P\in\mathbf{R}^{N}$.
We now define the \emph{solvability set} $\mathcal{D}$ as the set of $P\in\mathbf{R}^{N}$ such that \eqref{cp} admits a solution. Our second main result is a characterization of $\mathcal{D}$ in terms of the generalized effective Hamiltonian.
We prove that $\mathcal{D}=\{P\in\mathbf{R}^{N} \mid \bar{H}_{\infty}(P)<\underline{\sigma}\}$, where $\underline{\sigma}:=\min_{x\in\mathbf{T}^{N}}\sigma(x)$,
and that  $\bar{H}_{\infty}(P)$ is equal to the critical value of \eqref{cp} (Theorem \ref{thm:2}).
In the one-dimensional case, it turns out that $\mathcal{D}$ has a more explicit representation (Proposition \ref{onedim}).

We next present our homogenization results.
Let $u^{\varepsilon}$ be a viscosity solution of 
\begin{equation*}
\tag{HJ$_{\varepsilon}$}
\label{HJe}
\begin{cases}
\displaystyle
u_{t}^{\varepsilon}(x,t)+H\left(\frac{x}{\varepsilon},Du^{\varepsilon}(x,t)\right)=0 &\quad \text{in $\mathbf{R}^{N}\times(0,T)$},\\
u^{\varepsilon}(x,0)=u_{0}(x) &\quad \text{in $\mathbf{R}^{N}$}.
\end{cases}
\end{equation*}
Here, $\varepsilon>0$ is a parameter and $u_{0}:\mathbf{R}^{N}\to\mathbf{R}$ is a bounded and Lipschitz continuous initial datum. In our homogenization result (Theorem \ref{thm:hom}) we assume either
\begin{equation*}
(1)\quad \mathcal{D}=\mathbf{R}^{N} \quad \text{or} \quad (2)\quad m({\rm Lip}[u_{0}])<\underline{\sigma}/\overline{\sigma},
\end{equation*}
where $\overline{\sigma}:=\max_{x\in\mathbf{T}^{N}}\sigma(x)$ and ${\rm Lip}[u_{0}]$ stands for the Lipschitz constant of $u_{0}$. Then, we prove that $u^{\varepsilon}$ converges to the solution $u$ of the following problem locally uniformly in $\mathbf{R}^{N}\times[0,T)$ as $\varepsilon\to0$: 
\begin{equation*}
\tag{HJ}
\label{HJ}
\begin{cases}
u_{t}(x,t)+\bar{H}_{\infty}(Du(x,t))=0 &\quad \text{in $\mathbf{R}^{N}\times(0,T)$},\\
u(x,0)=u_{0}(x) &\quad \text{in $\mathbf{R}^{N}$}.
\end{cases}
\end{equation*}
The assumption (1) guarantees that the cell problem is solvable for every $P\in\mathbf{R}^{N}$. The proof is given by the half-relaxed limit method and the perturbed test function method provided by Evans \cite{E}. The assumption (2) is a sufficient condition that $\{u^{\varepsilon}\}_{\varepsilon>0}$ is equi-Lipschitz continuous. Since the cell problem may not have a solution for some $P\in\mathbf{R}^{N}$, we are not able to apply the perturbed test function method directly. We prove the homogenization result by reducing the original equation \eqref{cp} to the approximate equation \eqref{CPn} with a coercive Hamiltonian by using the equi-Lipschitz continuity of $\{u^{\varepsilon}\}_{\varepsilon>0}$.
We also show that, under the condition $\overline{\sigma}m(0)>\underline{\sigma}$, the solutions $u^{\varepsilon}$ do not converge to any function locally uniformly in $\mathbf{R}^{N}\times[0,T)$ as $\varepsilon\to0$ (Theorem \ref{non-hom}).

Our non-coercive Hamiltonian \eqref{eq:ourH} is originally derived by Yokoyama, Giga and Rybka \cite{YGR} to study the morphological stability of a faceted crystal.
Two functions $\sigma$ and $m$ represent the rate of supply of molecules and the dimensionless kinetic coefficient, respectively.
In \cite{GLM1} and \cite{GLM2} the authors study the large time behavior of a viscosity  solution of such non-coercive Hamilton-Jacobi equations.

We conclude this section with the physical explanation of the above homogenization problem and its result.
In this problem, we find an average growth of the faceted crystal with a (microscopic) heterogeneous supply of molecules.
As we will mention in Subsection 3.3, the cell problem does not have a solution under the condition $\overline{\sigma}m(0)\ge\underline{\sigma}$.
Thus, both the assumptions (1) and (2) imply $\overline{\sigma}m(0)<\underline{\sigma}$.
This inequality means that  the heterogeneity of the supply of molecules is somewhat small.
In this case the growth of the faceted crystal is described by \eqref{HJ} in view of Theorem \ref{thm:hom}.
We point out that the condition $\overline{\sigma}m(0)<\underline{\sigma}$ also appears in \cite{GLM2} to ensure the large time behaviour in the whole space.
On the other hand, if $\overline{\sigma}m(0)>\underline{\sigma}$, i.e., the heterogeneity is somewhat large, then the growth of the faceted crystal becomes complicated since homogenization fails (Theorem \ref{non-hom}).

In this paper, we show main theorems (Theorems \ref{thm:1} and \ref{thm:2}) under (H1)--(H3) for simplicity,
but it is possible to generalize a condition on a Hamiltonian and some generalizations are given as Remarks \ref{re:generalize} and \ref{re:generalize2}.

This paper is organized as follows.
Section 2 is devoted to preparation for the viscosity solutions and the critical values.
We study the cell problem in Section 3 and 4.
In Section 3, we present main theorems and prove them.
We also give a sufficient condition for $\mathcal{D}=\mathbf{R}^N$
and some properties of generalized effective Hamiltonians.
In Section 4, we focus on the one-dimensional cell problem and give a more explicit representation of $\mathcal{D}$.
Section 5 is concerned with an application to homogenization problems.
In Section 6 we extend the homogenization results for more general equations.

\section{Preliminaries}
In this section let $H:\mathbf{T}^{N}\times\mathbf{R}^{N}\to\mathbf{R}$ be a general continuous Hamiltonian.

Let ${\rm Lip}(\mathbf{T}^{N})$ denote the set of Lipschitz continuous functions on $\mathbf{T}^{N}$ and $\overline{B(x,r)}$ denote the closure of an open ball $B(x,r)$ of radius $r>0$ centered at a point $x$.

We consider Hamilton-Jacobi equations of the form
\begin{equation}
\label{eqhj}
H(x,Du(x))=0\quad \text{in $\mathbf{T}^{N}$}.
\end{equation}
In order to define viscosity solutions of \eqref{eqhj}, we recall notions of super- and subdifferentials. 
For a continuous function $u:\mathbf{T}^{N}\to\mathbf{R}$ and $x\in\mathbf{T}^{N}$, we set
\begin{equation*}
D^{+}u(x) := \Set{D\phi(x) | \phi\in C^{1}(\mathbf{T}^{N}),\quad \max_{\mathbf{T}^{N}}(u-\phi)=(u-\phi)(x)}.
\end{equation*}
We also define $D^- u(x)$ by replacing ``max'' by ``min'' in the above.

We call $u\in C(\mathbf{T}^{N})$ a \emph{viscosity subsolution} (resp.\ \emph{supersolution}) of \eqref{eqhj} if $H(\hat{x},p)\le0$ (resp.\ $H(\hat{x},p)\ge0$) for all $\hat{x}\in\mathbf{T}^{N}$ and $p\in D^{+}u(\hat{x})$ (resp.\ $p\in D^{-}u(\hat{x})$).
If $u\in C(\mathbf{T}^{N})$ is a viscosity sub- and supersolution of \eqref{eqhj},
we call it a \emph{viscosity solution} of \eqref{eqhj}.
The term ``viscosity'' is often omitted in this paper.

A pair of a function $u\in {\rm Lip}(\mathbf{T}^{N})$ and a constant $a\in\mathbf{R}$ satisfying \eqref{cp} is called a solution of \eqref{cp}. 
If such a constant $a$ is unique,
it is called the \emph{critical value} of \eqref{cp}. 
If there exists a critical value of the cell problem for every $P\in\mathbf{R}^{N}$, then we say that the cell problem is \emph{fully solvable}. 
When the cell problem is fully solvable, we are able to define a function $\bar{H}:\mathbf{R}^{N}\to\mathbf{R}$ by setting $\bar{H}(P)$ as the associated critical value. 
We call the function $\bar{H}$ an \emph{effective Hamiltonian} of $H$.

\begin{proposition}[Comparison principle for the cell problem]
\label{prop:cpcp}
Let $P\in\mathbf{R}^{N}$ and let $a, b \in \mathbf{R}$.
If there exist a subsolution $u \in {\rm Lip}(\mathbf{T}^{N})$ of \eqref{cp} and a supersolution $v \in {\rm Lip}(\mathbf{T}^{N})$ of $H(x,Dv(x)+P)=b$ in $\mathbf{T}^{N}$,
then $a \ge b$.
In particular, if $(u, c), (v, d) \in {\rm Lip}(\mathbf{T}^{N})\times\mathbf{R}$ are solutions of the cell problem \eqref{cp},
then $c = d$
and moreover
\begin{align*}
c
&= \inf\{ a \in \mathbf{R} \mid \text{there exists a subsolution of \eqref{cp}} \} \\
&= \sup\{ a \in \mathbf{R} \mid \text{there exists a supersolution of \eqref{cp}} \}.
\end{align*}
\end{proposition}

The proof is based on the comparison principle for \eqref{eq:dcp}
with a small $\delta>0$; see \cite{LPV, E}.
Here we do not need an extra continuity assumption on $H$
since $u$ and $v$ are now Lipschitz continuous.

\begin{proof}
Since $u$ and $v$ is bounded, we may assume that $u > v$ by adding a positive constant to $u$ if necessary.
Suppose by contradiction that $a < b$,
i.e.,
\begin{equation*}
H(x, D u+P) \le a < b \le H(x, D v+P)
\end{equation*}
in the viscosity sense.
We then see that
\begin{equation*}
\delta u+H(x, D u+P)
\le \frac{a+b}{2}
\le \delta v+H(x, D v+P).
\end{equation*}
The comparison principle implies that $u \le v$, which contradicts to $u > v$.
Therefore, $a \ge b$.

This observation implies
\begin{align*}
&\inf\{ a \in \mathbf{R} \mid \text{there exists a subsolution of \eqref{cp}} \} =: \overline{c} \\
&\quad \ge \sup\{ a \in \mathbf{R} \mid \text{there exists a supersolution of \eqref{cp}} \} =: \underline{c}.
\end{align*}
We then see that $c = d = \overline{c} = \underline{c}$
since $\overline{c} \le c \le \underline{c}$ and $\overline{c} \le d \le \underline{c}$ by the definitions.
\end{proof}

\begin{lemma}[Estimates of the critical value]
\label{est}
Let $P\in\mathbf{R}^{N}$ and let $(u,c)\in{\rm Lip}(\mathbf{T}^{N})\times\mathbf{R}$ be a solution of the cell problem \eqref{cp}.
Then, we have 
\begin{equation*}
\left.\begin{array}{l}
	\displaystyle \sup_{\phi \in C^{1}(\mathbf{T}^{N})}\inf_{x \in \mathbf{T}^{N}}H(x, D\phi(x)+P) \\
	\displaystyle \sup_{x \in \mathbf{T}^{N}}\sup_{p \in D^{+}u(x)}H(x, p+P) \\
\end{array}\right\}
\le c \le
\left\{\begin{array}{l}
	\displaystyle \inf_{\phi \in C^{1}(\mathbf{T}^{N})}\sup_{x \in \mathbf{T}^{N}}H(x, D\phi(x)+P), \\
	\displaystyle \inf_{x \in \mathbf{T}^{N}}\inf_{p \in D^{-}u(x)}H(x, p+P). \\
\end{array}\right.
\end{equation*}
\end{lemma}

\begin{proof}
The inequality $c \le \inf_{x \in \mathbf{T}^{N}}\inf_{p \in D^{-}u(x)}H(x, p+P)$ is trivial
since it is equivalent to the definition of a viscosity supersolution of \eqref{cp}.
Similarly, the inequality $\sup_{x \in \mathbf{T}^{N}}\sup_{p \in D^{+}u(x)}H(x, p+P) \le c$ holds
since it is equivalent to the definition of a viscosity subsolution of \eqref{cp}.

For a fixed $\phi \in C^{1}(\mathbf{T}^{N})$,
since $u-\phi$ is periodic and (Lipschitz) continuous,
we have $D\phi(\hat{x}) \in D^{-}u(\hat{x})$ at a minimum point $\hat{x}\in\mathbf{T}^{N}$ of $u-\phi$.
Thus
\begin{equation*}
\sup_{x \in \mathbf{T}^{N}}H(x, D\phi(x)+P)
\ge H(\hat{x}, D\phi(\hat{x})+P) \ge c,
\end{equation*}
which implies that $c \le \inf_{\phi \in C^{1}(\mathbf{T}^{N})}\sup_{x \in \mathbf{T}^{N}}H(x, D\phi(x)+P)$.
In a similar way, we see that $\sup_{\phi \in C^{1}(\mathbf{T}^{N})}\inf_{x \in \mathbf{T}^{N}}H(x, D\phi(x)+P) \le c$
by choosing a maximum point of $u-\phi$.
\end{proof}

\begin{remark}
It is worth to note that if the Hamiltonian $H = H(x, p)$ is convex in $p$ for each $x \in \mathbf{T}^{N}$ and satisfies the coercivity condition \eqref{eq:coercive},
then
\begin{equation*}
\inf_{\phi \in {\rm Lip}(\mathbf{T}^{N})}\sup_{x \in \mathbf{T}^{N}}\sup_{p \in D^{+}\phi(x)}H(x, p+P)
= \inf_{\phi \in C^{1}(\mathbf{T}^{N})}\sup_{x \in \mathbf{T}^{N}}H(x, D\phi(x)+P).
\end{equation*}
In particular, we have well-known formulas
\begin{align*}
c
&= \inf_{\phi \in C^{1}(\mathbf{T}^{N})}\sup_{x \in \mathbf{T}^{N}}H(x, D\phi(x)+P) \\
&= \inf_{\phi \in {\rm Lip}(\mathbf{T}^{N})}\sup_{x \in \mathbf{T}^{N}}\sup_{p \in D^{+}\phi(x)}H(x, p+P) \\
&= \sup_{x \in \mathbf{T}^{N}}\sup_{p \in D^{+}u(x)}H(x, p+P).
\end{align*}
We refer the reader to \cite{CIPP} or \cite[Subsection 4.2]{MT}
for details on such a kind of representation formulas of the critical value.
\end{remark}

We investigate the cell problem with a coercive Hamiltonian.

\begin{proposition}[\cite{LPV}]
\label{prop:ccp}
Assume \eqref{eq:coercive}.
Then, the cell problem \eqref{cp} is fully solvable.
\end{proposition}

\begin{proposition}[Properties of the effective Hamiltonian]
\label{prop:eh}
Assume \eqref{eq:coercive}.
\begin{itemize}
\item[(1)]
If there exists $L > 0$ such that $|H(x, p)-H(x, q)| \le L|p-q|$ for all $x \in \mathbf{T}^N$, $p, q \in \mathbf{R}^N$,
then $\bar{H}$ satisfies $|\bar{H}(P)-\bar{H}(Q)| \le L|P-Q|$ for all $P, Q \in \mathbf{R}^N$.
\item[(2)]
If $H(x, p) \le H(x, k p)$ for all $x \in \mathbf{T}^N$, $p \in \mathbf{R}^N$ and $k \ge 1$,
then $\bar{H}(P) \le \bar{H}(k P)$ for all $P \in \mathbf{R}^{N}$ and $k \ge 1$.
\item[(3)]
If $H(x, p) = H(x, -p)$ for all $x \in \mathbf{T}^N$ and $p \in \mathbf{R}^N$,
then $\bar{H}(P) = \bar{H}(-P)$ for all $P \in \mathbf{R}^{N}$.
\end{itemize}
\end{proposition}

\begin{proof}
(1)
Let $(u, \bar{H}(P))$ be a solution of \eqref{cp}. 
We observe
\begin{equation*}
H(x, D u+Q)-L|P-Q|
\le H(x, D u+P) = \bar{H}(P).
\end{equation*}
Thus, $u$ is a subsolution of
\begin{equation*}
H(x, D u+Q) = \bar{H}(P)+L|P-Q|.
\end{equation*}
By Proposition \ref{prop:cpcp}, we obtain that $\bar{H}(Q) \le \bar{H}(P)+L|P-Q|$.

(2)
Let $(u, \bar{H}(P))$ be a solution of \eqref{cp}.
We then see by the assumption that
\begin{equation*}
H(x, D(k u)+k P) \ge H(x, D u+P) = \bar{H}(P),
\end{equation*}
which means that $k u$ is a supersolution of $H(x, D v+k P) = \bar{H}(P)$.
Proposition \ref{prop:cpcp} implies $\bar{H}(k P) \ge \bar{H}(P)$.

(3)
Let $(u, \bar{H}(P))$ be a solution of \eqref{cp}. 
Then, since $H$ is even in the second variable, $(-u, \bar{H}(P))$ is a solution of 
\begin{equation*}
H(x, D v-P)=\bar{H}(P)\quad \text{in $\mathbf{T}^{N}$}.
\end{equation*}
Thus, we have $\bar{H}(P)=\bar{H}(-P)$.
\end{proof}

\section{The cell problem}

From now on, we study a Hamiltonian $H$ of the form \eqref{eq:ourH} with (H1)--(H3).
Define
\begin{equation*}
\overline{\sigma}:=\sup_{x\in\mathbf{T}^{N}}\sigma(x),\quad\underline{\sigma}:=\inf_{x\in\mathbf{T}^{N}}\sigma(x),\quad m_{0}:=m(0).
\end{equation*}
We note that (H3) ensures $m_{0}=\min_{r\in[0,\infty)}m(r)$.

\subsection{Main results}

For each $n\in\mathbf{N}$ let $H_{n}:\mathbf{T}^{N}\times\mathbf{R}^{N}\to\mathbf{R}$ be an approximating Hamiltonian of $H$ such that
\begin{itemize}
\item[(A1)]
$H_{n}$ is continuous on $\mathbf{T}^{N}\times\mathbf{R}^{N}$,
\item[(A2)]
$H_{n}$ satisfies the coercivity condition \eqref{eq:coercive},
\item[(A3)]
$\displaystyle\liminf_{n\to\infty}\inf_{\mathbf{T}^{N}\times\overline{B(0,R)}}(H-H_{n}) \ge 0$ for all $R>0$,
\item[(A4)]
$\displaystyle\limsup_{n\to\infty}\sup_{\mathbf{T}^{N}\times\mathbf{R}^{N}}(H-H_{n}) \le 0$.
\end{itemize}
By (A1) and (A2), for each $n\in\mathbf{N}$, the approximation cell problem \eqref{CPn} is fully solvable as noted in Proposition \ref{prop:ccp}.
Let $\bar{H}_{n}(P)$ be the critical value of \eqref{CPn} for $P \in \mathbf{R}^{N}$. 
We define a solvability set $\mathcal{D}$ by
\begin{equation*}
\mathcal{D}:=\{P\in\mathbf{R}^{N} \mid \text{\eqref{cp} admits a solution $(u,c)\in {\rm Lip}(\mathbf{T}^{N})\times\mathbf{R}$}\}.
\end{equation*}
We are now in a position to state our main theorems.

\begin{theorem}[Convergence of $\bar{H}_{n}$]\label{thm:1}
There exists a unique function $\bar{H}_{\infty}:\mathbf{R}^{N} \to \mathbf{R}$ such that, for any sequence $\{H_{n}\}_{n\in\mathbf{N}}$ satisfying (A1)--(A4), the following conditions hold:
\begin{align} 
&\liminf_{n\to\infty}\inf_{\overline{B(0,R)}}
(\bar{H}_{\infty}-\bar{H}_{n}) \ge 0 \quad 
\text{for all $R>0$,} \label{eq:conv1} \\
&\limsup_{n\to\infty}\sup_{\mathbf{R}^{N}}
(\bar{H}_{\infty}-\bar{H}_{n}) \le 0. \label{eq:conv2}
\end{align}
\end{theorem}

We call the function $\bar{H}_{\infty}$ 
a \emph{generalized effective Hamiltonian} of $H$.

\begin{theorem}[Characterization of the solvability set]\label{thm:2}
We have $\mathcal{D}=\{P\in\mathbf{R}^{N} \mid \bar{H}_{\infty}(P)<\underline{\sigma}\}$.
Moreover, if $P\in\mathcal{D}$, the critical value of \eqref{cp} is equal to $\bar{H}_{\infty}(P)$.
\end{theorem}

\subsection{The proof of pointwise convergence}

The proof of Theorem \ref{thm:1} consists of four steps.
We first prove in Step 1 that 
$\{\bar{H}_{n}(P)\}_{n\in\mathbf{N}}$ is a convergent sequence
for every $P\in\mathbf{R}^{N}$.
Then it is shown in Step 2 that
the limit is unique no matter how $\{H_{n}\}_{n\in\mathbf{N}}$ satisfying (A1)--(A4) is chosen.
In Step 3 we prove that the convergence is locally uniform 
when $\{H_{n}\}_{n\in\mathbf{N}}$ is monotone,
and finally, in Step 4, 
we derive \eqref{eq:conv1} and \eqref{eq:conv2} for a general approximation.
In this section, we shall show the first two steps.

Step 1.
Fix any $P\in\mathbf{R}^{N}$ and let $(u_{n},\bar{H}_{n}(P)) \in {\rm Lip}(\mathbf{T}^{N})\times\mathbf{R}$ be a solution of \eqref{CPn} for each $n\in\mathbf{N}$. We first show that $\{\bar{H}_{n}(P)\}_{n\in\mathbf{N}}$ is bounded from below. 
Indeed, taking a maximum point $x_{n}\in\mathbf{T}^{N}$ of $u_{n}$, we have
\begin{equation*}
H_{n}(x_{n},P)\le\bar{H}_{n}(P).
\end{equation*}
Since $H_{n}$ uniformly converges to $H$ on $\mathbf{T}^{N}\times\overline{B(0,|P|)}$, we see that
\begin{equation*}
H(x_{n},P)-1\le\bar{H}_{n}(P)
\end{equation*}
for sufficiently large $n$. 
Thus
\begin{equation*}
\left(\inf_{x\in\mathbf{T}^{N}}H(x,P)\right)-1\le\bar{H}_{n}(P),
\end{equation*}
which implies $\{\bar{H}_{n}(P)\}_{n\in\mathbf{N}}$ is bounded from below.

Fix $\varepsilon>0$. By (A4) there exists some $K\in\mathbf{N}$ such that 
\begin{equation}\label{eq1}
H-\frac{\varepsilon}{2}\le H_{n}\quad\text{on $\mathbf{T}^{N}\times\mathbf{R}^{N}$}
\end{equation}
for all $n\ge K$. Fix an arbitrary $n\ge K$. Recall that $u_{n}$ is a Lipschitz continuous function and set $L_{n}=|P|+{\rm Lip}[u_{n}]$. Then, it follows from (A3) that there exists some $M\ge n$ such that
\begin{equation}\label{eq2}
H_{m}-\frac{\varepsilon}{2}\le H\quad\text{on $\mathbf{T}^{N}\times\overline{B(0,L_{n})}$}
\end{equation}
for all $m\ge M$.
Combining \eqref{eq1} and \eqref{eq2}, we see that $u_{n}$ is a subsolution of
\begin{equation*}
H_{m}(x, D w+P)=\bar{H}_{n}(P)+\varepsilon\quad\text{in $\mathbf{T}^{N}$}.
\end{equation*}
By Proposition \ref{prop:cpcp}, we have
\begin{equation}\label{eq3}
\bar{H}_{m}(P)\le\bar{H}_{n}(P)+\varepsilon
\end{equation}
for all $m\ge M$. 
This inequality implies that $\{\bar{H}_{n}(P)\}_{n\in\mathbf{N}}$ is bounded from above. 
By taking $\limsup_{m\to\infty}$ and $\liminf_{n\to\infty}$, where $\limsup_{m\to\infty}$ should be operated first since $M$ depends on $n$, we have
\begin{equation*}
\limsup_{m\to\infty}\bar{H}_{m}\le\liminf_{n\to\infty}\bar{H}_{n}+\varepsilon.
\end{equation*}
Since $\varepsilon>0$ is arbitrary, $\bar{H}_{n}(P)$ converges to some value as $n\to\infty$.

Step 2.
We next prove that the limit of $\bar{H}_{n}(P)$ is independent of a choice of $\{H_{n}\}_{n\in\mathbf{N}}$ satisfying (A1)--(A4). 
Let $\{H_{n}\}_{n\in\mathbf{N}}$ and $\{H_{n}^{\prime}\}_{n\in\mathbf{N}}$ be two sequences of Hamiltonians satisfying (A1)--(A4). 
For each $P\in\mathbf{R}^{N}$, let $(u_{n},\bar{H}_{n}(P))$ and $(u_{n}^{\prime},\bar{H}_{n}^{\prime}(P))$ be, respectively, solutions of (CP$_{n}$) and
\begin{equation*}
H^{\prime}_{n}(x,Du_{n}^{\prime}+P)=\bar{H}_{n}^{\prime}(P)\quad\text{in $\mathbf{T}^{N}$}.
\end{equation*}
Consider a new sequence
\begin{equation*}
H_{1},H^{\prime}_{1},H_{2},H^{\prime}_{2},H_{3},H^{\prime}_{3},\cdots.
\end{equation*}
This satisfies (A3) and (A4), so that
\begin{equation*}
\bar{H}_{1}(P),\bar{H}^{\prime}_{1}(P),\bar{H}_{2}(P),\bar{H}^{\prime}_{2}(P),\bar{H}_{3}(P),\bar{H}^{\prime}_{3}(P),\cdots
\end{equation*}
has a limit $a\in\mathbf{R}$. 
Therefore
\begin{equation*}
a=\lim_{n\to\infty}\bar{H}_{n}(P)=\lim_{n\to\infty}\bar{H}^{\prime}_{n}(P)
\end{equation*}
since both $\{\bar{H}_{n}(P)\}_{n\in\mathbf{N}}$ and $\{\bar{H}^{\prime}_{n}(P)\}_{n\in\mathbf{N}}$ are subsequences.
We denote this common limit by $\bar{H}_{\infty}(P)$.

\begin{remark}\label{re:approx}
Thanks to the uniqueness of the pointwise limit we are able to take a specific approximation.
Let us take an approximating Hamiltonian $H_n$ of the form
\begin{equation*}
H_{n}(x,p)=\sigma(x)M_{n}(|p|),
\end{equation*}
where $M_{n}:[0,\infty)\to[m_{0},\infty)$ is an approximating function of $m$ such that
\begin{itemize}
\item[(B1)]
$M_{n}$ is Lipschitz continuous,
\item[(B2)]
$M_{n}(r) \to \infty$ as $r \to \infty$,
\item[(B3)]
there exists $\alpha_{n}\in\mathbf{R}$ such that
\begin{equation*}
M_{n}(r)=m(r)\quad\text{for $r\in[0,\alpha_{n}]$},\quad M_{n}(r)>m(r)\quad\text{for $r\in(\alpha_{n},\infty)$},
\end{equation*}
for each $n\in\mathbf{N}$ and $\alpha_{n}\to\infty$ as $n\to\infty$,
\item[(B4)]
$M_{n}(r)\ge M_{n^{\prime}}(r)$ for all $n^{\prime}\ge n$ and $r\in[0,\infty)$.
\end{itemize}
For instance, 
\begin{equation}\label{eq:Mn}
M_{n}(r)=\max\{m(r),Lr-n\}
\end{equation}
satisfies (B1)--(B4), 
where $L$ is the Lipschitz constant of $m$.
\end{remark}

\subsection{Properties of the  generalized effective Hamiltonian}

In this subsection we shall derive some properties of the generalized effective Hamiltonian.
These properties will improve the convergence of $\bar{H}_{n}$.

\begin{proposition}[Properties of the generalized effective Hamiltonian]
\label{prop:eeh}
We have
\begin{itemize}
\item[(1)]
$|\bar{H}_{\infty}(P)-\bar{H}_{\infty}(Q)| \le \overline{\sigma}L|P-Q|$ for all $P, Q \in \mathbf{R}^N$,
where $L$ is the Lipschitz constant of $m$,
\item[(2)]
$\bar{H}_{\infty}(k P)\ge\bar{H}_{\infty}(P)$ for all $P\in\mathbf{R}^{N}$ and $k\ge1$,
\item[(3)]
$\bar{H}_{\infty}(P)=\bar{H}_{\infty}(-P)$ for all $P\in\mathbf{R}^{N}$,
\item[(4)]
$\max\{\underline{\sigma}m(|P|),\overline{\sigma}m_{0}\}\le\bar{H}_{\infty}(P)\le\overline{\sigma}m(|P|)<\overline{\sigma}$ for all $P\in\mathbf{R}^{N}$.
\end{itemize}
\end{proposition}

\begin{proof}
Take $H_n$ as in Remark \ref{re:approx},
where we set $M_n$ by \eqref{eq:Mn}.
Let $\bar{H}_{n}$ be the effective Hamiltonian of $H_n$.
We then have
\begin{equation*}
|H_n(x, p)-H_n(x, q)| \le \overline{\sigma}L|p-q| \quad \text{for all $x \in \mathbf{T}^N$, $p, q \in \mathbf{R}^N$.}
\end{equation*}
Hence, Proposition \ref{prop:eh} (1) shows
\begin{equation*}
|\bar{H}_n(P)-\bar{H}_n(Q)| \le \overline{\sigma}L|P-Q| \quad \text{for all $P, Q \in \mathbf{R}^N$.}
\end{equation*}
Sending $n\to\infty$ yields the conclusion (1).

By a similar argument the properties (2)--(3) are verified from Proposition \ref{prop:eh}
since our coercive Hamiltonians $H_{n}$
satisfy the assumptions of Proposition \ref{prop:eh} (2)--(3).
The property (4) is a consequence of Lemma \ref{est}.
\end{proof}

\subsection{The proof of Theorem \ref{thm:1}}

As noted in Subsection 3.2, the proof consists of four steps and the first two steps have already been shown in Subsection 3.2.
We start from Step 3.

Step 3.
Assume that $\{H_{n}\}_{n\in\mathbf{N}}$ is monotone, i.e.,
$H_{n} \geq H_{n'}$ on $\mathbf{T}^{N}\times\mathbf{R}^{N}$ for all $n \leq n'$.
By this monotonicity we see that $\bar{H}_{n} \geq \bar{H}_{n'}$ if $n \leq n'$.
Indeed, a solution $u_{n}$ of \eqref{CPn} is always a subsolution of
\begin{equation*}
H_{n'}(x, D u_{n}+P) = \bar{H}_{n}(P).
\end{equation*}
for $n^{\prime}\ge n$.
Thus Proposition \ref{prop:cpcp} yields $\bar{H}_{n}(P) \geq \bar{H}_{n'}(P)$.
Since $\bar{H}_{\infty}$ is continuous in view of Proposition \ref{prop:eeh} (1),
Dini's lemma implies that 
$\bar{H}_{n}$ converges to $\bar{H}_{\infty}$ locally uniformly in $\mathbf{R}^{N}$ as $n \to \infty$.
(For the proof of Proposition \ref{prop:eeh} (1)
we only need a pointwise convergence of $\bar{H}_{n}$ to $\bar{H}_{\infty}$ and the uniqueness of $\bar{H}_{\infty}$.)

Step 4.
We shall show \eqref{eq:conv1} and \eqref{eq:conv2} for a general $\{H_{n}\}_{n\in\mathbf{N}}$.
Sending $m \to \infty$ in \eqref{eq3} of Step 1,
we obtain 
\begin{equation*}
\bar{H}_{\infty}(P)\le\bar{H}_{n}(P)+\varepsilon.
\end{equation*}
This inequality holds for all $\varepsilon >0$, $n \ge K$ and $P \in \mathbf{R}^N$, 
where $K$ does not depend on $P$.
Accordingly we have 
$\sup_{\mathbf{R}^N} (\bar{H}_{\infty}-\bar{H}_{n}) \le \varepsilon$,
and thus taking $\limsup_{n \to \infty}$ yields \eqref{eq:conv2}
since $\varepsilon >0$ is arbitrary.

To prove \eqref{eq:conv1} we define $\{H'_{n}\}_{n\in\mathbf{N}}$ by
$H'_n(x,p):=\sup_{m \ge n}H_m (x,p)$.
Then $\{H'_{n}\}_{n\in\mathbf{N}}$ is monotone and 
$H'_n \ge H_n$ on $\mathbf{T}^{N}\times\mathbf{R}^{N}$ for all $n$.
Also, $\{H'_{n}\}_{n\in\mathbf{N}}$ satisfies (A1)--(A4);
it is easy to see that (A2)--(A4) hold
while the continuity condition (A1) is due to Ascoli-Arzel\`{a} theorem,
which asserts that, for a compact set $K\subset\mathbf{R}^{N}$, a sequence $\{f_{n}\}_{n\in\mathbf{N}}\subset C(K)$ is uniformly bounded and equicontinuous if and only if every subsequence of $\{f_{n}\}$ has a uniformly convergent subsequence. We apply the if-part of this theorem to see that $\{H_{n}\}$ is equi-continuous on each compact set of $\mathbf{T}^{N}\times\mathbf{R}^{N}$ since $H_{n}\to H$  uniformly on the set.
Therefore the supremum $H^{\prime}_{n}$ is continuous.
From Step 3 it follows that
$\bar{H}'_{n}$ converges to $\bar{H}_{\infty}$ locally uniformly.
Therefore, using $\bar{H}'_n \ge \bar{H}_n$, we observe
\begin{equation*}
\liminf_{n\to\infty}\inf_{\overline{B(0,R)}}
(\bar{H}_{\infty}-\bar{H}_{n}) 
\ge \liminf_{n\to\infty}\inf_{\overline{B(0,R)}}
(\bar{H}_{\infty}-\bar{H}'_{n})=0.
\end{equation*}
The proof is now complete.

\begin{remark}\label{re:generalize}
Theorem \ref{thm:1} still holds for more general, continuous Hamiltonians
which are not necessarily of the form \eqref{eq:ourH}.
Indeed, the above proof works if we require $H$ to satisfy
\begin{equation}\label{eq:pLip}
|H(x,p)-H(x,q)| \le L|p-q| \quad \text{for some $L>0$,}
\end{equation}
which is used to guarantee Proposition \ref{prop:eeh} (1).
\end{remark}

\subsection{The proof of Theorem \ref{thm:2}}

We first prepare
\begin{proposition}\label{prop:est2}
Let $P \in \mathcal{D}$ and let $c \in \mathbf{R}$ be the critical value of \eqref{cp}.
Then,

\begin{equation*}
\left.\begin{array}{r}
	\underline{\sigma}m(|P|)\\
	\overline{\sigma}m_{0}
\end{array}\right\}
\le c\le
\overline{\sigma}m(|P|),
\quad c<\underline{\sigma}.
\end{equation*}
In particular, we have $\mathcal{D}=\emptyset$ if $\overline{\sigma}m_{0}\ge\underline{\sigma}.$
\end{proposition}

\begin{proof}
Taking $\phi \equiv 0$ in Lemma \ref{est} implies
\begin{align*}
c
&\le \inf_{\phi \in C^{1}(\mathbf{T}^{N})}\sup_{x \in \mathbf{T}^{N}}H(x, D\phi(x)+P)
\le \sup_{x\in\mathbf{T}^{N}}H(x, P)
= \overline{\sigma}m(|P|), \\
c
&\ge \sup_{\phi \in C^{1}(\mathbf{T}^{N})}\inf_{x \in \mathbf{T}^{N}}H(x, D\phi(x)+P)
\ge \inf_{x\in\mathbf{T}^{N}}H(x, P)
= \underline{\sigma}m(|P|). \\
\end{align*}

We next show $c < \underline{\sigma}$.
Take a solution $u \in {\rm Lip}(\mathbf{T}^{N})$ of \eqref{cp}.
For every $x \in A^{-} := \{ x \in \mathbf{T}^{N} \mid D^{-}u(x) \neq \emptyset \}$, take $p \in D^{-}u(x)$.
Since $|p| \le {\rm Lip}[u]$, we have
\begin{equation*}
\inf_{p \in D^{-}u(x)}H(x, p+P)
\le H(x, p+P)
\le \sigma(x)m({\rm Lip}[u]+|P|).
\end{equation*}
Therefore, by Lemma \ref{est},
\begin{align*}
c
&\le \inf_{x \in \mathbf{T}^{N}}\inf_{p \in D^{-}u(x)}H(x, p+P)
\le \inf_{x \in A^{-}}\inf_{p \in D^{-}u(x)}H(x, p+P) \\
&\le \inf_{x \in A^{-}}\sigma(x)m({\rm Lip}[u]+|P|).
\end{align*}
According to \cite[Lemma 1.8 (d)]{BC}, the set $A^{-}$ is dense in $\mathbf{T}^{N}$. 
Thus, we obtain $c \le \underline{\sigma}m({\rm Lip}[u]+|P|) < \underline{\sigma}$.
The proof of the inequality $\overline{\sigma}m_{0} \le c$ is easier.
\end{proof}

Let $\widehat{\mathcal{D}}:=\{P\in\mathbf{R}^{N} \mid \bar{H}_{\infty}(P)<\underline{\sigma}\}$. 
We note that $\mathcal{D}=\widehat{\mathcal{D}}=\emptyset$ when $\overline{\sigma}m_{0}\ge\underline{\sigma}$. Indeed, Proposition \ref{prop:est2} implies $\mathcal{D}=\emptyset$ and Proposition \ref{prop:eeh} (4) implies $\widehat{\mathcal{D}}=\emptyset$. 
We may hereafter assume that $\overline{\sigma}m_{0}<\underline{\sigma}$.
Let us take the special approximating Hamiltonian $H_n(x, p) = \sigma(x)M_n(|
p|)$ with the conditions (B1)--(B4) in Remark \ref{re:approx}.

\noindent\textbf{\underline{Proof  of $\mathcal{D}\supset\widehat{\mathcal{D}}$}}. 
We define
\begin{equation*}
\mathcal{D}_{\ell}:=\{P\in\mathbf{R}^{N} \mid \bar{H}_{\ell}(P)\le\underline{\sigma}m(\alpha_{\ell})\}.
\end{equation*}
for $\ell\in\mathbf{N}$. Note that $\cup_{\ell=0}^{\infty}\mathcal{D}_{\ell}=\widehat{\mathcal{D}}$. It is easy to check this equation since $\bar{H}_{\ell}(P)\to\bar{H}_{\infty}(P)$ as $\ell\to\infty$ with $\bar{H}_{\ell}(P)\ge\bar{H}_{\infty}(P)$ and $\underline{\sigma}m(\alpha_{\ell})\to\underline{\sigma}$ as $\ell\to\infty$ with $\underline{\sigma}m(\alpha_{\ell})<\underline{\sigma}$.

Therefore, if $\mathcal{D}\supset\mathcal{D}_{\ell}$ for every $\ell\in\mathbf{N}$, then we will have $\mathcal{D}\supset\cup_{\ell=0}^{\infty}\mathcal{D}_{\ell}=\widehat{\mathcal{D}}$.

Fix any $P\in\mathcal{D}_{\ell}$ and let $(u_{n},\bar{H}_{n}(P))$ be a solution of \eqref{CPn}. 
Note that $\bar{H}_{n}(P)$ is monotone decreasing with respect to $n$ by (B4) and Proposition \ref{prop:cpcp}. 
For each $n\in\mathbf{N}$ such that $n\ge\ell$,
\begin{equation*}
M_{n}(|Du_{n}+P|)=\frac{\bar{H}_{n}(P)}{\sigma(x)}\le\frac{\bar{H}_{\ell}(P)}{\underline{\sigma}}\le m(\alpha_{\ell})\quad\text{in $\mathbf{T}^{N}$}
\end{equation*}
in the viscosity sense. 
Note that the last inequality follows from $P\in\mathcal{D}_{\ell}$. 
Since $m\le M_{n}$ on $[0,\infty)$ and $m$ is strictly increasing, we have
\begin{equation*}
|Du_{n}(x)|\le\alpha_{\ell}+|P| \quad \text{in $\mathbf{T}^{N}$}
\end{equation*}
in the viscosity sense. 
Thus,
\begin{equation*}
\sup_{n\ge \ell}{\rm Lip}[u_{n}]\le\alpha_{\ell}+|P|<\infty.
\end{equation*}

Set $v_{n}(y):=u_{n}(y)-\min u_{n}$. 
Then, $\{v_{n}\}_{n\in\mathbf{N}}$ is uniformly bounded and equi-Lipschitz continuous in $\mathbf{T}^{N}$. 
Thus, by taking a subsequence if necessary, Ascoli-Arzel\`{a} theorem implies that $v_{n}$ uniformly converges to some Lipschitz continuous function $u$ in $\mathbf{T}^{N}$ as $n\to\infty$. 
Since $M_{n}$ converges to $m$ locally uniformly in $[0,\infty)$ by (B3) of $M_{n}$, the stability of viscosity solutions (see \cite{CIL}) implies that $(u,\bar{H}_{\infty}(P))$ is a solution of \eqref{cp}, which means that $P\in\mathcal{D}$. We get the desired inclusion $\mathcal{D}\supset\widehat{\mathcal{D}}$.

\noindent\underline{\textbf{\upshape Proof of $\mathcal{D}\subset\widehat{\mathcal{D}}$}}.  
Fix any $P\in\mathcal{D}$ and let $(u,c)\in$Lip$(\mathbf{T}^{N})\times\mathbf{R}$ be a solution of \eqref{cp}. 
The condition (B3) of $M_{n}$ implies
\begin{equation*}
M_{n}(r)=m(r)\quad \text{for all $r\le{\rm Lip}[u]+|P|$}
\end{equation*}
for sufficiently large $n$. 
Hence, $(u,c)$ is a solution of 
\begin{equation*}
\sigma(x)M_{n}(|Du+P|)=c \quad \text{in $\mathbf{T}^{N}$}.
\end{equation*}
Since $\bar{H}_{n}(P)$ is the critical value of the above problem,
we have $\bar{H}_{n}(P)=c$. 
Sending $n\to\infty$ yields $\bar{H}_{\infty}(P)=c$. 
Since $c<\underline{\sigma}$ by Proposition \ref{prop:est2}, we have $\bar{H}_{\infty}(P)<\underline{\sigma}$. 
Thus, 
$\mathcal{D}\subset\widehat{\mathcal{D}}$. 
The proof of Theorem \ref{thm:2} is complete.

\begin{remark}
By the last part of the proof, we see that for every $R>0$
there exists $N_R \in \mathbf{N}$ such that
$\bar{H}_{\infty}(P)=\bar{H}_{n}(P)$ for all $P \in \overline{B(0,R)}$ and $n \ge N_R$.
This is thanks to the conditions (B1)--(B4).
\end{remark}

\begin{remark}
Applying a method in the proof of Proposition \ref{prop:est2} to the approximating Hamiltonians $H_{n}$ above, we see $\max\{\underline{\sigma}M_{n}(|P|),\overline{\sigma}M_{n}(0)\}\le\bar{H}_{n}(P)\le\overline{\sigma}M_{n}(|P|)$. Since $M_{n}(0)=m(0)$, letting $P=0$ gives $\bar{H}_{n}(0)=\overline{\sigma}m_{0}$, so that $\bar{H}_{\infty}(0)=\overline{\sigma}m_{0}$. 
Thus Theorem \ref{thm:2} implies that $0 \in \mathcal{D}$ if $\overline{\sigma}m_{0}<\underline{\sigma}$.
Moreover, from the Lipschitz continuity of $\bar{H}_{\infty}$ 
(Proposition \ref{prop:eeh} (1))
it follows that $B(0,(\underline{\sigma}-\overline{\sigma}m_{0})/\overline{\sigma}L) \subset \mathcal{D}$,
where $L$ is the Lipschitz constant of $m$.
\end{remark}

\begin{remark}\label{re:generalize2}
A similar proof applies to more general Hamiltonians.
Let $H$ be a Hamiltonian satisfying \eqref{eq:pLip}. 
We define
\begin{equation*}
h(\rho):=\inf_{x \in \mathbf{T}^N}\inf_{|p| \ge \rho}H(x,p), \quad
h_{\infty}:=\sup_{\rho \ge 0}h(\rho),
\end{equation*}
and assume 
\begin{itemize}
\item[(H4)]
$\displaystyle{\inf_{x \in \mathbf{T}^N}\sup_{|p| \le \rho}H(x,p)< h_{\infty}}$ for all $\rho \ge 0$.

\end{itemize}
Then it turns out that 
$\mathcal{D}=\{P\in\mathbf{R}^{N} \ | \ \bar{H}_{\infty}(P)<h_{\infty}\}$.
We shall give a sketch of the proof of this generalization.

We first show that the critical value $c$ of \eqref{cp} satisfies $c<h_{\infty}$.
In a similar way to the proof of Proposition \ref{prop:est2},
we see
\[ c \le \inf_{x \in A^-} \inf_{p \in D^-u(x)}H(x, p+P)
\le \inf_{x \in A^-} \sup_{|p| \le \rho}H(x, p) \]
with $\rho:=\mathrm{Lip}[u]+|P|$,
where $u \in \mathrm{Lip}(\mathbf{T}^N)$ is a solution of \eqref{cp}
and $A^-= \{ x \in \mathbf{T}^N \mid D^- u(x) \neq \emptyset \}$.
Since $\sup_{|p| \le \rho}H(\cdot, p)$ is continuous 
by the compactness of $\overline{B(0,\rho)}$ (see Lemma \ref{lem:supf})
and since $A^-$ is dense in $\mathbf{T}^N$,
using (H4), we estimate
\[ c \le \inf_{x \in \mathbf{T}^N} \sup_{|p| \le \rho}H(x, p)<h_{\infty}. \]
Next, we see that 
$\{ p \in \mathbf{R}^{N} \ | \ H(x,p)\le \tau \ \text{for some} \ x \in \mathbf{T}^N \}$
is bounded for every $\tau<h_{\infty}$.
Indeed, if there were some sequence $\{ (x_j,p_j) \}_{j \in \mathbf{N}}$ such that $|p_j| \to \infty$ as $j \to \infty$,
we would have
$h(|p_j|) \le H(x_j,p_j) \le \tau < h_{\infty}$,
which is a contradiction since $\sup_{j \in \mathbf{N}} h(|p_j|)<h_{\infty}$.

Define $\widehat{\mathcal{D}}:=\{P\in\mathbf{R}^{N} \ | \ \bar{H}_{\infty}(P)<h_{\infty}\}$,
and take an approximate Hamiltonian $H_n$ as $H_n(x,p)=\max\{ H(x,p),|p|-n \}$.
To prove $\mathcal{D} \supset \widehat{\mathcal{D}}$
we set $\mathcal{D}_{\ell}:=\{ P \in \mathbf{R}^{N} \ | \ \bar{H}_{\ell}(P) \le \tau_{\ell} \}$,
where $\{ \tau_{\ell} \}_{\ell \in \mathbf{N}}$ is a sequence such that 
$\tau_{\ell} < h_{\infty}$ and $\tau_{\ell} \to h_{\infty}$ as $\ell \to \infty$.
Then $\bigcup_{\ell=1}^{\infty}\mathcal{D}_{\ell}=\widehat{\mathcal{D}}$.
Fix $\ell \in \mathbf{N}$.
For every $P \in \mathcal{D}_{\ell}$ and $n \ge l$,
a solution $(u_n,\bar{H}_n(P))$ of \eqref{CPn} satisfies
\begin{equation*}
H(x,Du_n +P) \le H_n(x,Du_n +P)=\bar{H}_n(P)\le \bar{H}_{\ell}(P) \le \tau_{\ell}.
\end{equation*}
Since $\tau_{\ell}<h_{\infty}$, 
we have $\sup_{n\ge \ell}{\rm Lip}[u_{n}]<\infty$.
Ascoli-Arzel\`{a} theorem ensures that 
$u_n-\min u$ subsequently converges to some $u$,
and thus $(u,\bar{H}_{\infty}(P))$ solves \eqref{cp}.
The proof of $\mathcal{D} \subset \widehat{\mathcal{D}}$ is easier.
Indeed, by the choice of $H_n$, a solution $(u,c)$ of \eqref{cp} 
is also a solution of \eqref{CPn} for $n$ sufficiently large,
and therefore $\bar{H}_{\infty}(P)=\bar{H}_n(P)=c<h_{\infty}$.
\end{remark}

\subsection{A sufficient condition for the fully solvability}

Applying the result in Theorem \ref{thm:2},
we give a sufficient condition which guarantees that
\eqref{cp} is fully solvable, i.e., $\mathcal{D}=\mathbf{R}^n$.

\begin{theorem}\label{thm:3}
Assume $\overline{\sigma}m_{0}<\underline{\sigma}$.
Let $P \in \mathbf{R}^N$ and
assume that there exists $\psi \in C^1(\mathbf{T}^N)$ such that $D\psi =-P$ 
on $\{ \sigma \neq \underline{\sigma} \}$.
Then $P \in \mathcal{D}$.
\end{theorem}

If there exists such a $\psi$ for every $P \in \mathbf{R}^N$,
then \eqref{cp} is fully solvable.
A simple condition for the existence of $\psi$
will be given after the proof; see Remark \ref{rem:scond}.

\begin{proof}
We take $H_n$ as in Remark \ref{re:approx}.
By the representation of $\mathcal{D}$ obtained in Theorem \ref{thm:2},
the proof is completed by showing that $\bar{H}_n(P)<\underline{\sigma}$
for $n \in \mathbf{N}$ sufficiently large.
To this end, we use the estimate 
\begin{equation*}
\bar{H}_n(P) \le \inf_{\phi \in C^1(\mathbf{T}^N)} \sup_{x \in \mathbf{T}^N}
H_n(x,D\phi (x)+P)
\end{equation*}
in Lemma \ref{est}.
Choosing $\phi=\psi$, 
where $\psi$ is the function in our assumption, we see
\begin{equation}\label{eq:sHn}
\bar{H}_n(P) \le \sup_{x \in \mathbf{T}^N}
H_n(x,D\psi (x)+P).
\end{equation}
On $\{ \sigma \neq \underline{\sigma} \}$ we compute
\begin{equation*}
H_n(x,D\psi (x)+P)
=H_n(x,0)
=\sigma(x) m_0
\le \overline{\sigma}m_0
<\underline{\sigma}.
\end{equation*}
For $x \in \mathbf{T}^N$ such that $\sigma(x)=\underline{\sigma}$, 
we have
\begin{equation*}
H_n(x,D\psi (x)+P)
=\underline{\sigma} M_n (|D\psi (x)+P|).
\end{equation*}
We now set $r_0:=\max_{x \in \mathbf{T}^N} |D\psi (x)+P|<\infty$
and choose $n$ large so that $M_n(r)=m(r)$ for all $r \le r_0$.
Then 
\begin{equation*}
H_n(x,D\psi (x)+P)
\le \underline{\sigma} m(r_0)
<\underline{\sigma}.
\end{equation*}
Consequently, \eqref{eq:sHn} implies $\bar{H}_n(P)<\underline{\sigma}$.
\end{proof}

\begin{remark}\label{rem:scond}
If $\overline{\{ \sigma \neq \underline{\sigma} \}} \subset (0,1)^N$,
then there exists $\psi$ in Theorem \ref{thm:3} for every $P \in \mathbf{R}^N$.
Indeed, letting $A \subset (0,1)^N$ be an open set such that
$\overline{\{ \sigma \neq \underline{\sigma} \}} \subset A$ and 
$\overline{A} \subset (0,1)^N$, we are able to construct a function $\psi \in C^1(\mathbf{T}^N)$ so that 
$\psi(x)=-\langle P,x \rangle$ for $x \in \{ \sigma \neq \underline{\sigma} \}$ 
and $\psi(x)=0$ for $x \not \in A$.
The existence of such a $\psi$ is due to Whitney's extension theorem;
see, e.g., \cite[Section 6.5, Theorem 1]{EG}.
Let us briefly check the assumption in \cite{EG}. 
Let $f:[0,1]^N \to \mathbf{R}$ and $d:[0,1]^N \to \mathbf{R}^N$
be continuous functions such that
$f(x)=-\langle P,x\rangle$, $d(x)=-P$ 
on $K_1:=\overline{\{ \sigma \neq \underline{\sigma} \}}$ and 
$f(x)=d(x)=0$ on $K_2:=[0,1]^N \setminus A$.
Also, define $K:=K_1 \cup K_2$ and $\delta_0:=\mathrm{dist}(K_1,K_2)>0$.
If $x,y \in K$ satisfy $|x-y|<\delta_0$,
we have $x,y \in K_1$ or $x,y \in K_2$,
and hence 
$R(y,x):=(f(y)-f(x)-\langle d(x),y-x \rangle)/|x-y|=0$.
The theorem is thus applicable.
\end{remark}

\begin{remark}
The existence of $\psi$ in Theorem \ref{thm:3}
is not a necessary condition for $P \in \mathcal{D}$.
In Example \ref{exa2} (1), where we consider the one-dimensional case,
the cell problem is fully solvable, but there is no such periodic $\psi$ for $P \neq 0$
because $\sigma$ attains a minimum at one point.
\end{remark}

We conclude this section by giving 
continuity results for the supremum of continuous functions.
The results are fundamental, 
but we give them for the reader's convenience.

\begin{lemma}\label{lem:supf}
Let $\{ f_{\lambda} \}_{\lambda \in \Lambda} \subset C(X)$, 
where $X \subset \mathbf{R}^N$ and $\Lambda$ is an index set.
Assume that $\{ f_{\lambda} \}$ is uniformly bounded on $X$, 
and set $f:=\sup_{\lambda \in \Lambda}f_{\lambda}$.
Then the following hold:
\begin{itemize}
\item[(1)]
$f$ is lower semicontinuous on $X$.
\item[(2)]
Assume that $\{ f_{\lambda} \}$ is equicontinuous on $X$, i.e.,
$\displaystyle{\lim_{y \to x} \sup_{\lambda \in \Lambda} 
|f_{\lambda}(y)-f_{\lambda}(x)|=0}$.
Then $f$ is upper semicontinuous on $X$.
\item[(3)]
Assume that $\Lambda \subset \mathbf{R}^M$ is compact
and that $\displaystyle{\lim_{(y,\mu) \to (x,\lambda)}f_{\mu}(y)=f_{\lambda}(x)}$
for all $x \in X$ and $\lambda \in \Lambda$.
Then $f$ is upper semicontinuous on $X$.
\end{itemize}
\end{lemma}

\begin{proof}
Fix $x \in X$ and $\varepsilon >0$.
For any $y \in X$ there exists $\lambda(\varepsilon, y) \in \Lambda$
such that $f(y) \le f_{\lambda(\varepsilon, y)}(y) + \varepsilon$
by the definition of $f$.
Then we have
\[ f_{\lambda(\varepsilon, x)}(y)
-f_{\lambda(\varepsilon, x)}(x)-\varepsilon
\le f(y)-f(x) 
\le f_{\lambda(\varepsilon, y)}(y)
+\varepsilon -f_{\lambda(\varepsilon, y)}(x). \]
Sending $y \to x$ in the first inequality implies (1) while 
(2) follows from the second inequality and the equicontinuity.
Here note that, in the latter case, the index $\lambda(\varepsilon, y)$
changes as $y \to x$,
and so we need to estimate the right-hand side by
$\sup_{\lambda \in \Lambda} |f_{\lambda}(y)-f_{\lambda}(x)|+\varepsilon$.
When $\Lambda$ is compact,
we may assume that $\lambda(\varepsilon, y)$ is a convergent sequence as $y \to x$.
Therefore, using the second inequality again, we obtain (3).
\end{proof}

\section{One-dimensional cell problem}
In this section we investigate the cell problem in one dimension. 
In this case the solvability set $\mathcal{D}$ has a more explicit representation. 
We first rewrite \eqref{cp} as
\begin{equation}\label{eq:eikonal}
|u^{\prime}(x)+P|=f_{a}(x)\quad\text{in $\mathbf{T}$},
\end{equation}
where
\begin{equation*}
f_{a}(x):=m^{-1}\left(\frac{a}{\sigma(x)}\right).
\end{equation*}
Here, $m^{-1}:[m_{0},1)\to[0,\infty)$ is the inverse function of $m$, and $f_{a}$ is well-defined as a $[0,\infty)$-valued function if $\overline{\sigma}m_{0}\le a<\underline{\sigma}$. 
We now set $m^{-1}(1)=\infty$. 
Then, $f_{\underline{\sigma}}$ is a $[0,\infty]$-valued function. 
Note that $a\mapsto f_{a}(x)$ is increasing for every $x\in\mathbf{T}$. 

The authors of \cite{LPV} consider
\begin{equation}\label{eq:lpv}
|u^{\prime}(x)+P|^{2}-V(x)=a\quad \text{in $\mathbf{T}$}, 
\end{equation}
as an example of the cell problem in one dimension. 
Here, $V$ is a continuous function on $\mathbf{T}$ such that $\min_{\mathbf{T}}V=0$. 
According to \cite{LPV}, for each $P\in\mathbf{R}$, the critical value $c$ of \eqref{eq:lpv} is given by
\begin{equation}\label{eq:repc}
	c=\begin{cases}
	0 \quad &\text{if $|P|\le\int_{0}^{1}\sqrt{V(z)}dz$,}\\
	a \quad   \text{such that $|P|=\int_{0}^{1}\sqrt{V(z)+a}dz$, $a\ge0$},&\text{otherwise.}
	\end{cases}
\end{equation}
As an analogue of this formula, we establish

\begin{proposition}\label{onedim}
\begin{itemize}
\item[(1)]
If $\overline{\sigma}m_{0}\ge\underline{\sigma}$, then $\mathcal{D}=\emptyset$.
\item[(2)]
If $\overline{\sigma}m_{0}<\underline{\sigma}$, then
\begin{equation*}
	\mathcal{D}=\begin{cases}
		(-\int_{0}^{1}f_{\underline{\sigma}}(z)dz,\int_{0}^{1}f_{\underline{\sigma}}(z)dz)\quad &\text{if $f_{\underline{\sigma}}\in L^{1}(0,1)$},\\
		\mathbf{R}\quad &\text{otherwise}.
	\end{cases}
\end{equation*}
\end{itemize}
Moreover, the critical value $c$ is given by
\begin{equation}\label{eq:cv}
	c=\begin{cases}
	\overline{\sigma}m_{0}&\text{if $|P|\le\int_{0}^{1}f_{\overline{\sigma}m_{0}}(z)dz$,}\\
	a \quad \text{such that $|P|=\int_{0}^{1}f_{a}(z)dz$}&\text{otherwise.}
	\end{cases}
\end{equation}
\end{proposition}

\begin{proof}
(1) This is obvious by Proposition \ref{prop:est2}.

(2)  We set $\widetilde{\mathcal{D}}=(-\int_{0}^{1}f_{\underline{\sigma}}(z)dz,\int_{0}^{1}f_{\underline{\sigma}}(z)dz)$.
When $f_{\underline{\sigma}}\notin L^{1}(0,1)$, we read $\widetilde{\mathcal{D}}=\mathbf{R}$. We first prove $\mathcal{D}\supset\widetilde{\mathcal{D}}$.
To do this, take $P\in\widetilde{\mathcal{D}}$.
What we have to do is to find $u\in{\rm Lip}(\mathbf{T})$ such that $(u,c)$ is a solution of \eqref{eq:eikonal}, where $c$ is the constant in \eqref{eq:cv}.

When $|P|\le\int_{0}^{1}f_{\overline{\sigma}m_{0}}(z)dz$, we set
\begin{equation*}
	u(x)=\begin{cases}
		\displaystyle\int_{x_{0}}^{x}f_{\overline{\sigma}m_{0}}(z)dz-P x\quad &\text{for $x\in[x_{0},x_{1}]$,}\\
		\displaystyle\int_{x}^{x_{0}+1}f_{\overline{\sigma}m_{0}}(z)dz+P(1-x)\quad&\text{for $x\in[x_{1},x_{0}+1]$.}
	\end{cases}
\end{equation*}
Here, $x_{0}\in[0,1]$ and $x_{1}\in[x_{0},x_{0}+1]$ are points such that
\begin{equation*}
f_{\overline{\sigma}m_{0}}(x_{0})=0,\quad\int_{x_{0}}^{x_{1}}f_{\overline{\sigma}m_{0}}(z)dz=\int_{x_{1}}^{x_{0}+1}f_{\overline{\sigma}m_{0}}(z)dz+P.
\end{equation*}
We regard $u$ as a function on $\mathbf{T}$ by extending it periodically. 
Then, it is easy to see that $u$ is a solution of \eqref{eq:eikonal}. 

When $|P|\ge\int_{0}^{1}f_{\overline{\sigma}m_{0}}(z)dz$, for $c$ chosen by \eqref{eq:cv}, we set
\begin{equation*}
u(x)=\sign(P)\int_{0}^{x}f_{c}(z)dz-P x \quad \text{for $x\in\mathbf{R}$.}
\end{equation*}
Note that $u$ is a $\mathbf{Z}$-periodic function since, by the definition of $c$,
\begin{equation*}
\sign(P)\int_{0}^{1}f_{c}(z)dz-P=0.
\end{equation*}
Then, it is easy to see that $u$ is a solution of \eqref{eq:eikonal}. 
Therefore, we have obtained $\mathcal{D}\supset\widetilde{\mathcal{D}}$.

We next show the reverse inclusion $\mathcal{D}\subset\widetilde{\mathcal{D}}$. 
Let $P\in\mathcal{D}$ and take a solution $(u,c)$ of \eqref{eq:eikonal}, then
\begin{equation*}
|P|=\left|\int_{0}^{1}(u'(z)+P)dz\right|\le\int_{0}^{1}|u'(z)+P|dz\le\int_{0}^{1}f_{c}(z)dz<\int_{0}^{1}f_{\underline{\sigma}}(z)dz.
\end{equation*}
The first equality follows from the periodicity of $u$. Thus, $P\in\widetilde{\mathcal{D}}$ and so the proof is complete.
\end{proof}

\begin{remark}
The representation of the critical value \eqref{eq:cv}
is also obtained via the formula \eqref{eq:repc} given in \cite{LPV}.
In fact, $a$ is a critical value of \eqref{cp}
if and only if the critical value $c_a$ of 
\begin{equation*}
|u^{\prime}(x)+P|=f_{a}(x)+c_a \quad\text{in $\mathbf{T}$}
\end{equation*}
is equal to $0$.
It is easily seen that the condition $c_a=0$ yields \eqref{eq:cv}.
\end{remark}

When $\sigma$ attains a minimum on some interval $[a,b]$ with $a<b$, it is easily seen that $f_{\underline{\sigma}}$ is not integrable since $f_{\underline{\sigma}}=+\infty$ on $[a,b]$. Consequently, \eqref{eq:eikonal} is fully solvable by Proposition \ref{onedim}.
If $\sigma(x)=\underline{\sigma}$ at only one point $x\in\mathbf{T}$, the integrability of $f_{\underline{\sigma}}$ depends on $\sigma$ and $m$ as the next examples indicate.

\begin{example}
Let us consider \eqref{eq:eikonal} with
\begin{equation*}
m(r)=\frac{1}{2}\frac{r}{1+r}+\frac{1}{2}\quad(r\in[0,\infty)),
\quad\sigma(x)=x^{\alpha}(1-x)^{\alpha}+\beta\quad(x\in[0,1]),
\end{equation*}
where $\alpha,\beta>0$. 
We note that $\overline{\sigma}m_{0}<\underline{\sigma}$ holds when $\beta>1/4^{\alpha}$. 
Since
\begin{equation*}
f_{\underline{\sigma}}(x)=\frac{\underline{\sigma}}{2}\frac{1}{\sigma(x)-\underline{\sigma}}-1=\frac{\underline{\sigma}}{2}\frac{1}{x^{\alpha}(1-x)^{\alpha}}-1,
\end{equation*} 
the integrability of $f_{\underline{\sigma}}$ is determined by the choice of $\alpha>0$.
\end{example}

\begin{example}\label{exa2}
Set
\begin{equation*}
\sigma(x)=
\begin{cases}
x+\frac{3}{2}\quad&(0\le x<\frac{1}{2}),\\
-x+\frac{5}{2}\quad&(\frac{1}{2}\le x<1).
 \end{cases}
\end{equation*}

(1)  We let
\begin{equation*}
m(r)=\frac{1}{2}\frac{r}{1+r}+\frac{1}{2}.
\end{equation*}
We extend $\sigma$ periodically to $\mathbf{R}$ and still denote it by $\sigma$. 
Note that $\overline{\sigma}m_{0}<\underline{\sigma}$ holds. 
Since
\begin{equation*}
f_{\underline{\sigma}}(x)=\frac{\underline{\sigma}}{2}\frac{1}{\sigma(x)-\underline{\sigma}}-\frac{1}{2},
\end{equation*}
we observe that
\begin{equation*}
\int_{0}^{1}f_{\underline{\sigma}}(z)dz=2\int_{0}^{1/2}f_{\underline{\sigma}}(z)dz=\underline{\sigma}\int_{0}^{1/2}\frac{1}{z}dz-\frac{1}{2}=\infty.
\end{equation*}
Thus, $\mathcal{D}=\mathbf{R}$, i.e., the cell problem is fully solvable.

(2)  We next study 
\begin{equation*}
m(r)=\frac{1}{2}\tanh r+\frac{1}{2}\quad(r\in[0,\infty)).
\end{equation*}
Note that $\overline{\sigma}m_{0}<\underline{\sigma}$. Since
\begin{equation*}
f_{\underline{\sigma}}(x)=\frac{1}{2}\log\left(\frac{\underline{\sigma}}{\sigma(x)-\underline{\sigma}}\right)=\frac{1}{2}\{\log\underline{\sigma}-\log(\sigma(x)-\underline{\sigma})\},
\end{equation*}
we observe that
\begin{equation*}
\int_{0}^{1}f_{\underline{\sigma}}(z)dz=2\int_{0}^{1/2}f_{\underline{\sigma}}(z)dz=1+\frac{1}{2}\log6.
\end{equation*}
Therefore, $\mathcal{D}=\left(-1-\frac{1}{2}\log6,1+\frac{1}{2}\log6\right)$. 
\end{example}

\begin{proposition}
We have
\begin{equation*}
\bar{H}_{\infty}(P)>\bar{H}_{\infty}(Q)\quad \text{for all $P,Q\in\mathcal{D}$ such that $|P|>|Q|\ge\int_{0}^{1}f_{\overline{\sigma}m_{0}}(z)dz$.}
\end{equation*}
\end{proposition}

\begin{proof}
By \eqref{eq:cv}, we observe
\begin{equation*}
\int_{0}^{1}\{f_{\bar{H}_{\infty}(P)}(z)-f_{\bar{H}_{\infty}(Q)}(z)\}dz=|P|-|Q|>0,
\end{equation*}
which implies that $\bar{H}_{\infty}(P)>\bar{H}_{\infty}(Q)$.
\end{proof}

\section{Application to homogenization problems}

We present our homogenization result for the equation \eqref{HJe}
with the Hamiltonian \eqref{eq:ourH} satisfying (H1)--(H3).
Here, $u_{0}:\mathbf{R}^{N}\to\mathbf{R}$ is a bounded and Lipschitz continuous initial datum.
We remark that there exists a unique bounded solution $u^{\varepsilon}\in C(\mathbf{R}^{N}\times[0,T))$ of \eqref{HJe}. Similarly, there exists a unique bounded solution $u\in C(\mathbf{R}^{N}\times[0,T))$ of \eqref{HJ}.
Indeed, the comparison principle holds for a viscosity sub- and supersolution (see \cite{CIL}). 
This yields uniqueness of solutions. 
Existence is a consequence of Perron's method (see \cite{I}).

\begin{theorem}[Homogenization result]\label{thm:hom}
Assume either
\begin{equation*}
(1)\ \ \mathcal{D}=\mathbf{R}^{N}\quad{\rm or}\quad(2)\ \ m({\rm Lip}[u_{0}])<\underline{\sigma}/\overline{\sigma}.
\end{equation*}
Then the solution $u^{\varepsilon}$ of \eqref{HJe} converges to the solution $u$ of \eqref{HJ} locally uniformly in $\mathbf{R}^{N}\times[0,T)$ as $\varepsilon\to0$.
\end{theorem}

\begin{proof}[Proof of Theorem \ref{thm:hom} under the assumption (1)]

Recall that, for each $P\in\mathcal{D}$, $\bar{H}_{\infty}(P)$ is the critical value of \eqref{cp} from Theorem \ref{thm:2}.
As we mentioned in Introduction the assumption (1) means that the cell problem is fully solvable,
and so the conclusion follows from the same argument
as in \cite{E} involving the perturbed test function method.
Here we do not need equi-Lipschitz continuity of $\{ u_{\varepsilon} \}$
since the half-relaxed limit method works for our equation;
see \cite[Proof of Theorem 4.4]{E}.
To be more precise, it turns out that
the upper- and lower half-relaxed limits
\begin{align*}
\overline{u}(x,t) &:= \lim_{\delta \to 0}\sup\{ u^{\varepsilon}(y,s) \mid (y,s) \in B(x, \delta)\times(t-\delta, t+\delta),\ \varepsilon < \delta \}, \\
\underline{u}(x,t) &:= \lim_{\delta \to 0}\inf\{ u^{\varepsilon}(y,s) \mid (y,s) \in B(x, \delta)\times(t-\delta, t+\delta),\ \varepsilon < \delta \}
\end{align*}
are, respectively, a sub- and supersolution of (HJ),
so that the comparison principle ensures that these two limits
are equal to the solution $u$.
This implies locally uniform convergence to $u$.

\end{proof}

We shall hereafter prove Theorem \ref{thm:hom} under the assumption (2).

\begin{proposition}[Regularity of the solution of \eqref{HJe}]
\label{prop:reg}
Assume (2) in Theorem \ref{thm:hom}. Then, the solutions $u^{\varepsilon}$ of \eqref{HJe} satisfy
\begin{equation*}
|u^{\varepsilon}(x, t) - u^{\varepsilon}(x, s)| \le L|t-s|, \quad |u^{\varepsilon}(x, t) - u^{\varepsilon}(y, t)| \le K|x-y|
\end{equation*}
for all $x, y \in \mathbf{R}^N, t, s \in [0, T)$ with the constants
\begin{equation*}
L := \overline{\sigma}m({\rm Lip}[u_0]) < \infty, \quad K := m^{-1}\left(\frac{\overline{\sigma}}{\underline{\sigma}}m({\rm Lip}[u_0])\right) < \infty.
\end{equation*}
\end{proposition}

We omit the proof since this proposition is verified by the same argument as in \cite[Appendix A]{GH}.
We point out that \cite[Proposition 3.17]{GH} holds under the assumption $R_+(m) < \infty$ even if the Hamiltonian does not satisfy the coercivity condition (H$_{R+}$).

We give two different proofs of Theorem \ref{thm:hom} under the assumption (2).

\begin{proof}[Proof I of Theorem \ref{thm:hom} under the assumption (2)]
By Proposition \ref{prop:reg}, Ascoli-Arzel\`{a} theorem implies that $u^{\varepsilon}$ subsequently converges to some Lipschitz continuous function $u$ locally uniformly in $\mathbf{R}^{N}\times[0,T)$ as $\varepsilon\to0$.

We prove that $u$ is a supersolution of \eqref{HJ}. 
The proof is based on the perturbed test function method (see \cite{E}). 
Let $(x_{0},t_{0})\in\mathbf{R}^{N}\times(0,T)$ and $\phi\in C^{1}(\mathbf{R}^{N}\times(0,T))$ such that $u-\phi$ has a strict local minimum at $(x_{0},t_{0})$. 
Suppose that
\begin{equation*}
\phi_{t}(x_{0},t_{0})+\bar{H}_{\infty}(D\phi(x_{0},t_{0}))=:-\theta<0.
\end{equation*} 
We take $\{H_{n}\}_{n\in\mathbf{N}}$ as Remark \ref{re:approx} and let $\bar{H}_{n}$ be the effective Hamiltonian of \eqref{CPn}. 
Since $\bar{H}_{n}$ converges to $\bar{H}_{\infty}$, we have
\begin{equation}
\phi_{t}(x_{0},t_{0})+\bar{H}_{n}(D\phi(x_{0},t_{0}))\le-\frac{\theta}{2}
\end{equation}
for sufficiently large $n$. 
On the other hand, by the Lipschitz continuity of $u^{\varepsilon}$  and (B3), we see that it is a solution of 
\begin{equation}\label{eq:n-eq}
w_{t}(x,t)+H_{n}\left(\frac{x}{\varepsilon},Dw(x,t)\right)=0
\end{equation}
in $\mathbf{R}^{N}\times(0,T)$ for sufficiently large $n$. We hereafter fix $n$ satisfying the above two conditions.

Set
\begin{equation*}
\phi^{\varepsilon}_{n}(x,t):=\phi(x,t)+\varepsilon v_{n}\left(\frac{x}{\varepsilon}\right),
\end{equation*}
where $v_{n}$ is a solution of \eqref{CPn}. 
By the same argument as in \cite{E}, we see that $\phi^{\varepsilon}_{n}$ is a subsolution of \eqref{eq:n-eq} in $B(x_{0},r)\times(t_{0}-r,t_{0}+r)$ for sufficiently small $r>0$.
The comparison principle for \eqref{HJe} implies a contradiction (see \cite{E}) and so $u$ is a supersolution of \eqref{HJ}.

Similarly, it is proved that $u$ is a subsolution of \eqref{HJ}, and therefore, $u$ is a unique solution of \eqref{HJ}.
Consequently, $u^{\varepsilon}$ converges to $u$ locally uniformly in $\mathbf{R}^{N}\times[0,T)$ as $\varepsilon\to0$ without taking subsequences.
\end{proof}

\begin{proof}[Proof II of Theorem \ref{thm:hom} under the assumption (2)]
Recall that $\{u^{\varepsilon}\}_{\varepsilon > 0}$ is equi-Lipschitz continuous in view of Proposition \ref{prop:reg}
and therefore subsequently converges to some $u$ locally uniformly in $\mathbf{R}^{N}\times[0,T)$ as $\varepsilon\to0$.
Take $\{H_{n}\}_{n\in\mathbf{N}}$ as Remark \ref{re:approx}.
By (B3) and the equi-Lipschitz continuity of $\{u^{\varepsilon}\}_{\varepsilon > 0}$,
we have
\begin{equation*}
u^{\varepsilon}_t(x, t)+H_{n}\left(\frac{x}{\varepsilon}, D u^{\varepsilon}(x, t)\right)=0 \quad \text{in $\mathbf{R}^{N}\times(0, T)$}
\end{equation*}
for all $n \in \mathbf{N}$ large enough and  all $\varepsilon > 0$.
We now apply the homogenization result for coercive Hamiltonians \cite{E} to see that $u^{\varepsilon}$ converges to the solution $w_{n}$ of
\begin{equation*}
	\begin{cases}
		\displaystyle w_{t}(x,t)+\bar{H}_{n}\left(D w(x,t)\right)=0\quad&\text{in $\mathbf{R}^{N}\times(0,T)$,}\\
		w(x,0)=u_{0}(x)\quad&\text{in $\mathbf{R}^{N}$}
	\end{cases}
\end{equation*}
locally uniformly in $\mathbf{R}^{N}\times[0,T)$.
Since $u$ is a limit of a subsequence,
it turns out that $w_{n} \equiv u$.
Since $\bar{H}_{n}$ converges to $\bar{H}_{\infty}$ locally uniformly,
the stability result for viscosity solutions yields the conclusion that $u$ is a viscosity solution of \eqref{HJ}.
\end{proof}

\begin{remark}
The main difference between two proofs is the order of limits of $\varepsilon$ and $n$.
We point out that Proof I does not require the locally uniform convergence of $\bar{H}_{n}$.
However, we need the equi-Lipschitz continuity of $\{u^{\varepsilon}\}_{\varepsilon > 0}$ in both proofs
in order to ensure that $u^{\varepsilon}$ is a solution of the approximate equation.
\end{remark}

\begin{theorem}[Non-homogenization result]\label{non-hom}
Assume that $\overline{\sigma}m_{0}> \underline{\sigma}$. Let $u^{\varepsilon}$ be the solutions of \eqref{HJe}.  Then, $u^{\varepsilon}$ does not have a locally uniformly convergent limit in $\mathbf{R}^{N}\times[0,T)$ as $\varepsilon\to0$.
\end{theorem}

\begin{proof}
Set
\begin{equation*}
u_{-}^{\varepsilon}(x,t):=u_{0}(x)-\sigma\left(\frac{x}{\varepsilon}\right)t,\quad u_{+}^{\varepsilon}(x,t):=u_{0}(x)-\sigma\left(\frac{x}{\varepsilon}\right)m_{0}t
\end{equation*}
for $(x,t)\in\mathbf{R}^{N}\times[0,T)$. 
Then, we see that $u_{-}^{\varepsilon}$ and $u_{+}^{\varepsilon}$ are a subsolution and a supersolution of (HJ$_{\varepsilon}$) respectively. 
By the comparison principle, the solution $u^{\varepsilon}$ satisfies
\begin{equation}\label{eq:est3}
u_{-}^{\varepsilon}(x,t)\le u^{\varepsilon}(x,t)\le u_{+}^{\varepsilon}(x,t)
\end{equation}
for all $(x,t)\in\mathbf{R}^{N}\times[0,T)$ and $\varepsilon>0$.

Let $\overline{u}$ and $\underline{u}$ be the upper half-relaxed limit and the lower half-relaxed limit of $u^{\varepsilon}$, respectively. 
Moreover, let $\overline{u}_{-}$ and $\underline{u}_{+}$ be the upper half-relaxed limit of $u_{-}^{\varepsilon}$ and the lower half-relaxed limit of $u_{+}^{\varepsilon}$, respectively. 
Then, we have
\begin{equation*}
\overline{u}_{-}(x,t)=u_{0}(x)-\underline{\sigma}t\quad{\rm and}\quad\underline{u}_{+}(x,t)=u_{0}(x)-\overline{\sigma}m_{0}t.
\end{equation*}
Thus, by \eqref{eq:est3} and the assumption $\overline{\sigma}m_{0}>\underline{\sigma}$, we have
\begin{equation}\label{eq:est4}
\underline{u}(x,t)\le\underline{u}_{+}(x,t)<\overline{u}_{-}(x,t)\le\overline{u}(x,t)
\end{equation}
for all $(x,t)\in\mathbf{R}^{N}\times[0,T)$. 
Therefore, $\overline{u}$ and $\underline{u}$ are different and so we conclude that $u^{\varepsilon}$ does not  converge to any functions locally uniformly as $\varepsilon\to0$.
\end{proof}



\begin{remark}
When $\overline{\sigma}m_{0}=\underline{\sigma}$, we do not know whether or not $u^{\varepsilon}$ has a limit as $\varepsilon\to0$. However, by \eqref{eq:est4}, we see that the limit of $u^{\varepsilon}$ should be $u_{0}(x)-\underline{\sigma}t(=u_{0}(x)-\overline{\sigma}m_{0}t)$ if it exists.
\end{remark}

\section{Generalization}

Our homogenization results can be extended for more general equations of the form
\begin{equation}
u^{\varepsilon}_t(x,t)
+H \left( x,\frac{x}{\varepsilon}, 
u^{\varepsilon}(x,t),D u^{\varepsilon}(x,t) \right)=0 
\quad \mbox{in} \ \mathbf{R}^N \times (0,T).
\label{gHJe}
\end{equation}
Here 
$H=H(x,y,u,p):\mathbf{R}^N \times \mathbf{T}^N \times \mathbf{R} \times \mathbf{R}^N \to \mathbf{R}$
is Lipschitz continuous in 
$\mathbf{R}^N \times \mathbf{T}^N \times (-L, L) \times B(0,L)$
for every $L>0$
and non-decreasing in $u$.
They guarantee the comparison principle;
similar assumptions can be seen in \cite{E}.
The corresponding cell problem is 
\begin{equation}
H \left( x,y,u,D v(y)+P \right) = a
\quad \text{in $\mathbf{T}^N$,}
\label{gCP}
\end{equation}
where the unknown is $(v,a) \in \mathrm{Lip}(\mathbf{T}^N) \times \mathbf{R}$
and $(x,u) \in \mathbf{R}^N \times \mathbf{R}$ is fixed.
Define $\mathcal{D}_{x,u}$ as the set of $P \in \mathbf{R}^N$
such that \eqref{gCP} admits a solution $(v,a)$ for a given $(x,u)$.
For homogenization of \eqref{gHJe} we assume either
\[ (1) \ \ \mathcal{D}_{x,u}=\mathbf{R}^N \ \mbox{for all $(x,u) \in \mathbf{R}^N \times \mathbf{R}$}
\quad \mbox{or} \quad 
(2) \ \ \sup_{\varepsilon>0} \mathrm{Lip}[u^{\varepsilon}] < \infty. \]

Choose $H_n(x,y,u,p) =\max\{ H(x,y,u,p), |p|-n \}$,
which is a coercive Hamiltonian approximating $H$.
Since $(x,u)$ is fixed in cell problems, 
a similar method in this paper gives 
a generalized effective Hamiltonian
$\bar{H}_{\infty}(x,u,P)$ as the limit of $\bar{H}_n(x,u,P)$.
(Here we do not pursue generalization of approximation to $H$
and study only a homogenization problem.
Also, in this case $\bar{H}_{\infty}$ is just the infimum of $\bar{H}_n$.)
According to \cite[Lemma 2.2]{E},
$\bar{H}_n$ possesses the same regularity and monotonicity properties as $H_n$,
and thus so is $\bar{H}_{\infty}$.
Moreover, since $\bar{H}_n$ is monotone in $n$, 
Dini's lemma ensures that $\bar{H}_n$ converges to $\bar{H}_{\infty}$ locally uniformly.
One is now able to show homogenization results for \eqref{gHJe} with the same argument as in three proofs of Theorem \ref{thm:hom} above.

\section*{Acknowledgments}
The authors thank the anonymous referee for his or her careful reading of the manuscript and valuable comments.
The work of the first author was supported by Grant-in-aid for Scientific Research of JSPS Fellows No.\ 23-4365 and No.\ 26-30001.
The work of the second author was supported by a Grant-in-Aid for JSPS Fellows No.\ 25-7077 and the Program for Leading Graduate Schools, MEXT, Japan.
The work of the third author was supported by the Program for Leading Graduate Schools, MEXT, Japan.





\begin{thebibliography}{9}


\bibitem{AB}
O.~Alvarez, M.~Bardi, {Singular perturbations of nonlinear degenerate parabolic PDEs: a general convergence result}, Arch. Ration. Mech. Anal. 170 (2003), 17--61.
\bibitem{AL}
M.~Arisawa, P.-L.~Lions, {On ergodic stochastic control}, Comm. Partial Differential Equations 23 (1998), 2187--2217.
\bibitem{BC}
M.~Bardi, I.~Capuzzo-Dolcetta, {Optimal control and viscosity solutions of Hamilton-Jacobi-Bellman equations. With appendices by Maurizio Falcone and Pierpaolo Soravia. Systems \& Control: Foundations \& Applications}, Birkh\"{a}user Boston, Inc., Boston, MA, 1997.
\bibitem{BT}
M.~Bardi, G.~Terrone, {On the homogenization of some non-coercive Hamilton-Jacobi-Isaacs equations},
Commun. Pure Appl. Anal. 12 (2013), 207--236.
\bibitem{B}
G.~ Barles, {Some homogenization results for non-coercive Hamilton-Jacobi equations}, Calc. Var. Partial Differential Equations 30 (2007), 449--466.
\bibitem{BW}
I.~Birindelli, J.~Wigniolle, {Homogenization of Hamilton-Jacobi equations in the Heisenberg group}, Commun. Pure Appl. Anal. 2 (2003), 461--479.
\bibitem{CM}
L.~A.~Caffarelli, R.~Monneau, {Counter-example in three dimension and homogenization of geometric motions in two dimension}, Arch. Ration. Mech. Anal. 212 (2014), 503--574.
\bibitem{C}
P.~Cardaliaguet, {Ergodicity of Hamilton-Jacobi equations with a noncoercive nonconvex Hamiltonian in $\mathbb{R}^{2}/\mathbb{Z}^{2}$}, Ann. Inst. H. Poincar\'{e} Anal. Non Lin\'{e}aire 27 (2010), 837--856.
\bibitem{CLS}
P.~Cardaliaguet, P.-L.~Lions, P.~E.~Souganidis, {A discussion about the homogenization of moving interfaces}, J. Math. Pures Appl. (9) 91 (2009), 339--363.
\bibitem{CNS}
P.~Cardaliaguet, J.~Nolen, P.~E.~Souganidis, {Homogenization and enhancement for the $G$-equation}, Arch. Ration. Mech. Anal. 199 (2011), 527--561.
\bibitem{CIPP}
G.~Contreras, R.~Iturriaga, G.P.~Paternain, M.~Paternain, {Lagrangian graphs, minimizing measures and Ma\~{n}\'{e}'s critical values}, Geom. Funct. Anal. 8 (1998), 788--809.
\bibitem{CIL}
M.~G.~Crandall, H.~Ishii, P.-L.~Lions, {User's guide to viscosity solutions of second order partial differential equations}, Bull. Amer. Math. Soc. (N.S.) 27 (1992), 1--67.
\bibitem{E}
L.~C.~Evans, {Periodic homogenisation of certain fully nonlinear partial differential equations}, Proc. Roy. Soc. Edinburgh Sect. A 120 (1992), 245--265.
\bibitem{EG}
L.~C.~Evans, R.~F.~Gariepy, {Measure theory and fine properties of functions}, Studies in Advanced Mathematics, CRC Press, Boca Raton, FL, 1992.
\bibitem{GH}
Y.~Giga, N.~Hamamuki, {Hamilton-Jacobi equations with discontinuous source terms}, Comm. Partial Differential Equations 38 (2013), 199--243.
\bibitem{GLM1}
Y.~Giga, Q.~Liu, H.~Mitake, {Large-time asymptotics for one-dimensional Dirichlet problems for Hamilton-Jacobi equations with noncoercive Hamiltonians}, J. Differential Equations 252 (2012), 1263--1282.
\bibitem{GLM2}
Y.~Giga, Q.~Liu, H.~Mitake, {Singular Neumann problems and large-time behavior of solutions of noncoercive Hamilton-Jacobi equations}, Trans. Amer. Math. Soc. 366 (2014), 1905--1941.
\bibitem{Go}
D.~A.~Gomes, {Hamilton-Jacobi methods for vakonomic mechanics}, NoDEA Nonlinear Differential Equations Appl. 14 (2007), 233--257.
\bibitem{I}
H.~Ishii, {Perron's method for Hamilton-Jacobi equations}, Duke Math. J. 55 (1987), 369--384. 
\bibitem{LPV}
P.-L.~Lions, G.~Papanicolaou, S.~R.~S.~Varadhan, {Homogenization of Hamilton-Jacobi equations}, unpublished.
\bibitem{MT}
H.~Mitake, H.~V.~Tran, {Homogenization of weakly coupled systems of Hamilton-Jacobi equations with fast switching rates}, Arch. Ration. Mech. Anal. 211 (2014), 733--769.
\bibitem{S}
B.~Stroffolini, {Homogenization of Hamilton-Jacobi equations in Carnot groups}, ESAIM Control Optim. Calc. Var. 13 (2007), 107--119 (electronic).
\bibitem{YGR}
E.~Yokoyama, Y.~Giga, P.~Rybka, {A microscopic time scale approximation to the behavior of the local slope on the faceted surface under a nonuniformity in supersaturation}, Phys. D 237 (2008), 2845--2855.
\end{thebibliography}


\end{document}